\documentclass[12pt,reqno]{amsart}
\usepackage{hyperref}
\usepackage{amsmath,amsfonts,amscd,amsthm,amssymb}
\usepackage{tikz,tikz-cd}
\usepackage{commath}
\usepackage[margin=1in]{geometry}
\usepackage{derivative}
\usepackage{enumitem}
\usepackage{changepage}
\usepackage{color,soul}

\definecolor{myurlcolor}{rgb}{0,0.6,0.2}
\hypersetup{colorlinks, linkcolor=myurlcolor, citecolor=myurlcolor, urlcolor=myurlcolor}

\newtheorem{thmalph}{Theorem}
\newtheorem{lemalph}[thmalph]{Lemma}
\newtheorem{dfnalph}[thmalph]{Definition}

\newtheorem{lemma}{Lemma}[section]
\newtheorem{proposition}[lemma]{Proposition}

\newtheorem{Main Lemma}[lemma]{Main Lemma}
\newtheorem*{theorem}{Theorem}
\theoremstyle{remark}				

\newtheoremstyle{acknowledge}{}{}{\itshape}{}{\bfseries}{:}{ }{}
\theoremstyle{acknowledge}
\newtheorem*{acknowledgment}{Acknowledgment}
 	
\numberwithin{equation}{section}	
\let\oldref\ref						\renewcommand{\ref}[1]{(\oldref{#1})}

\begin{document}
\title{Random 3-Manifolds Have No Totally Geodesic Submanifolds}
\author{Hasan M. El-Hasan}
\address{Hasan M. El-Hasan, Deptartment of Mathematics, University of
California, Riverside, Riverside, Ca 92521}
\email{helh001@ucr.edu}
\urladdr{https://sites.google.com/ucr.edu/helh001/}
\author{Frederick Wilhelm}
\thanks{Both authors gratefully acknowledge the support of the NSF, via
Award DMS 2203686.}
\address{Fred Wilhelm, Deptartment of Mathematics, University of California
Riverside, Riverside, Ca 92521. }
\email{fred@math.ucr.edu}
\urladdr{https://sites.google.com/site/frederickhwilhelmjr/home}
\date{April 2024}

\begin{abstract}
Murphy and the second author showed that a generic closed Riemannian
manifold has no totally geodesic submanifolds, provided it is at least four
dimensional. Lytchak and Petrunin established the same thing in dimension 3.
For the higher dimensional result, the \textquotedblleft generic
set\textquotedblright\ is open and dense in the $C^{q}$--topology for any $%
q\geq 2.$ In Lytchak and Petrunin's work, the \textquotedblleft generic
set\textquotedblright\ is a dense $G_{\delta }$ in the $C^{q}$--topology for
any $q\geq 2.$ Here we show that the set of such metrics on a compact $3$
--manifold contains a set that is open and dense in the $C^{q}$--topology
for any $q\geq 3.$
\end{abstract}

\maketitle

In his magnum opus, Spivak writes that it

\bigskip

\begin{quotation}
\textquotedblleft \textit{seems rather clear that if one takes a Riemannian
manifold }$\left( N,\left\langle \cdot ,\cdot \right\rangle \right) $ \textit{%
 `at random', then it will not have any totally geodesic
submanifolds of dimension $>1$. But I must admit that I don't know of any
specific example of such a manifold.\textquotedblright } (\cite{spivak}, p.
39)\bigskip
\end{quotation}

Murphy and the second author proved that this is indeed the case for compact
manifolds with dimension $\geq 4.$\bigskip

\begin{thmalph}[\protect\cite{MurphWilh}]
\label{murphWilh}Let $M$ be a compact, smooth manifold of dimension $\geq 4$%
. For any finite $q\geq 2$, the set of Riemannian metrics on $M$ with no
nontrivial immersed totally geodesic submanifolds contains a set that is
open and dense in the $C^{q}$--topology.
\end{thmalph}

Lytchak and Petrunin confirmed Spivak's intuition in dimension $3$. Indeed,
a consequence of their main theorem in \cite{LytPetr} is

\begin{theorem}[\protect\cite{LytPetr}]
Let $M $ be a compact, smooth $3 $--manifold. For any finite $q\geq 2$, let $%
\mathcal{M}^{q}\left( M\right) $ be the space of Riemannian metrics on $M $
with the $C^q $--topology. Then $\mathcal{M}^{q}\left( M\right) $ contains a
dense $G_{\delta }$, $\mathcal{G} $, so that for all $g\in \mathcal{G} $, $%
\left(M,g\right) $ has no immersed totally geodesic surfaces.
\end{theorem}

Here we refine this result by showing

\begin{thmalph}
\label{main thm}Let $M$ be a compact, smooth manifold of dimension $3$. For
any finite $q\geq 3$, the set of Riemannian metrics on $M$ with no immersed
totally geodesic surfaces contains a set that is open and dense in the $%
C^{q} $-topology.
\end{thmalph}

Motivated by Perelman's solution of Thurston's geometrization conjecture (%
\cite{Perel1}, \cite{Perel2}, \cite{Perel3}), it is also natural to ask
about the existence of totally geodesic surfaces in compact $3$--manifolds
that have one of Thurston's eight canonical geometries. This problem is
attacked algorithmically for a special family of surfaces in cusped
hyperbolic 3--manifolds by Basillo, Lee, and Malionek in \cite{BLM}.

In the math overflow post \cite{bryant}, Bryant frames the question of
whether a generic manifold has a totally geodesic hypersurface in terms of
the curvature tensor. In the case of dimension $3$, and when the Ricci
tensor has distinct eigenvalues, he identifies the only three tangent planes
at a given point that can possibly be tangent to a totally geodesic surface.

We will take an approach that is more directly connected to \cite{MurphWilh}
. For starters we let $\mathcal{G}r\left( M\right) $ be the Grassmannian
bundle of planes tangent to $M.$ We then define an invariant $G:\mathcal{G}
r\left( M\right) \longrightarrow \mathbb{R}$ that vanishes on $P\subset 
\mathcal{G}r\left( M\right) $ whenever $P$ is tangent to a totally geodesic
surface. We will then show how to deform a given Riemannian $3$--manifold so
that $G$ is nowhere $0.$

\begin{dfnalph}
\label{G Generic}Given $v\in TM,$ let $R_{v}=R(\cdot ,v)v:T_{p}M\rightarrow
T_{p}M$ be the corresponding Jacobi operator, and let $\nabla _{\cdot }R_{v}$
be the covariant derivative of $R_{v}.$ Define the \textbf{generic plane}
operator, 
\begin{equation*}
G:\mathcal{G}r\left( M\right) \longrightarrow \mathbb{R}
\end{equation*}%
by 
\begin{equation*}
G\left( P\right) :=\max_{\left\{ v,w\in P\text{ }|\text{ }|v|=\left\vert
w\right\vert =1\right\} }\left\{ \left\vert R_{v}\left( w\right) ^{P\perp
}\right\vert ,\left\vert \left( \left( \nabla _{v}R_{v}\right) \left(
w\right) \right) ^{P\perp }\right\vert \right\} ,
\end{equation*}%
where the superscript $^{P\perp }$ denotes the component perpendicular to $%
P. $ If $G\left( P\right) \neq 0,$ then we say that $P$ is $G$--\textbf{%
generic. }

If $G\left( P\right) =0,$ then we say that $P$ is $G$--\textbf{\ rigid. }We
say that $\left( M,g\right) $ is $G$--\textbf{generic}, provided all $P\in 
\mathcal{G}r\left( M\right) $ are generic.
\end{dfnalph}

Note that if $P$ is $G$--generic, then $P$ is not tangent to any totally
geodesic submanifold. Thus $G$--rigid planes are analogous to the
\textquotedblleft Partially Geodesic\textquotedblright\ planes of \cite%
{MurphWilh}, and Theorem \ref{main thm} is a consequence of the following
result.

\begin{thmalph}
\label{G generic them}Let $M$ be a compact, smooth $3$--manifold. For any
finite $q\geq 3,$ the set of $G$--generic Riemannian metrics on $M$ is open
and dense in the $C^{q}$--topology.
\end{thmalph}

Since $\mathcal{G}r\left( M\right) $ is compact and $G$ is continuous in the 
$C^{q}$--topology for any finite $q\geq 3$, the set of $G$--generic metrics
is open in the $C^{q}$--topology. Thus it suffices to show that the set of $%
G $--generic Riemannian metrics on $M$ is dense. More specifically it
suffices to show the

\bigskip

\noindent \textbf{Density Assertion: }\emph{Let }$\left( M,g\right) $\emph{\
be a compact, smooth Riemannian }$3$\emph{--manifold. Given any finite }$%
q\geq 3$\emph{\ and }$\xi >0$\emph{, there is a }$G$--\emph{generic
Riemannian metric }$\tilde{g}$\emph{\ on }$M$\emph{\ that satisfies} 
\begin{equation*}
\left\vert \tilde{g}-g\right\vert _{C^{q}}\leq \bigskip \xi .
\end{equation*}%
To prove the Density Assertion we will repeatedly apply the following result.

\begin{lemalph}
\label{lem-C}Let $\{g_{s}\}_{s\geq 0}$ be a smooth family of Riemannian
metrics on a \emph{\ }compact, smooth\emph{\ }$3$\emph{--}manifold $M$ with
corresponding generic plane operator $G^{s}.$ Let $\mathcal{U}$ be an open
subset of the Grassmannian $\mathcal{G}\left( M\right) .$ Let $\mathcal{R}$
be the set of $G$--rigid planes for $g_{0}$ in the closure $\overline{%
\mathcal{U}}$ of $\mathcal{U}.$ Suppose that there are $c,s_{0}$, and a
neighborhood $\mathcal{V}$ of $\mathcal{R}$ so that for every $P\in \mathcal{%
V},$ and every $s\in (0,s_{0})$, 
\begin{equation}
G^{s}\left( P\right) >cs.  \label{lem-C-inequality}
\end{equation}%
Then, for all sufficiently small $s$, every plane of $\overline{\mathcal{U}}$
is $G^{s}$--generic.
\end{lemalph}

In particular, when $\mathcal{U}=\mathcal{G}r\left( M\right) $, Lemma \ref%
{lem-C} gives a criterion for a global deformation to a metric that is $G$
--generic.

\begin{proof}
Since $\overline{\mathcal{U\setminus V}}$ is compact, there is a $\delta >0$
such that for all $P\in \overline{\mathcal{U\setminus V}},$ 
\begin{equation*}
G^{0}\left( P\right) >\delta .
\end{equation*}%
By continuity, there is an $s_{1}>0$ so that for all $P\in \overline{%
\mathcal{U\setminus V}}$ and all $s\in \left( 0,s_{1}\right) ,$ 
\begin{equation*}
G^{s}\left( P\right) >\delta .
\end{equation*}%
Combining this with Inequality \ref{lem-C-inequality}, it follows that for
all sufficiently small $s$, every plane in $\overline{\mathcal{U}}$ is $%
G^{s} $--generic.
\end{proof}

With one minor modification, the proof we present here gives an alternative
proof of the special case of Theorem \ref{murphWilh}, when $q\geq 3.$ The
reader who is familiar with \cite{MurphWilh} might also recognize other
similarities between the proof we present here and the one in that paper. In
fact, a quick summary of our proof is to use the ideas of \cite{MurphWilh}
with the addition of $\left\vert \left( \nabla _{v}R_{v}\right) \left(
w\right) ^{P\perp }\right\vert $ in Definition \ref{G Generic}.

Since we suspect that these comments alone will not be convincing to many
readers, we detail the rest of the proof below. We have prioritized having a
simple, clear, and self-contained exposition over attempting to avoid some
textual overlap with \cite{MurphWilh}. Having acknowledged this debt to \cite%
{MurphWilh}, we make no further mention of it.

The proof of the Density Assertion begins by establishing a local version of
it, which is the main lemma of the paper (\ref{local-main} below). For the
idea behind the proof of the main lemma, suppose that a plane $P\subset
T_{p}M$ is $G$--rigid with orthonormal basis $\left\{ v,T,N\right\} \subset
T_{p}M,$ with $v,T$ tangent to $P,$ and $n\perp P$. We perturb $\left\langle
n,T\right\rangle $ by a function $f$ that has either a large second or third
derivative in the $v$-direction on a neighborhood of $p$. Then either $%
R(T,v)v$ or $\left( \nabla _{v}R\right) \left( T,v\right) v$ will have a
large component in the $n$-direction. Thus, $P$ is no longer $G$ --rigid.

Section \ref{notation section} establishes the notations and conventions. In
Section \ref{local construct section}, we state and prove the main lemma (%
\ref{local-main}), which as explained above, is a local version of the
Density Assertion. In Section \ref{Global section}, we explain how to
combine Lemma \ref{lem-C} and the main lemma (\ref{local-main}) to prove the
Density Assertion and hence Theorems \ref{main thm} and \ref{G generic them}.

\begin{acknowledgment}
We are grateful to Paula Bergen for copy editing this manuscript.
\end{acknowledgment}

%---------------------------------------------------------------------------------------------------------------------------------%

\section{Notation\label{notation section}}

Let $(M,g)$ be a compact Riemannian manifold of dimension $3$. The
Riemannian connection, curvature tensor, and Christoffel symbols are denoted
by $\nabla $, $R$, and $\Gamma $.

Given local coordinates $\{x_{i}\}_{i=1}^{3}$, we write $\partial _{i}$ to
denote the partial derivative operator, $\frac{\partial }{\partial x_{i}}$.

Let $\mathcal{G}r\mathcal{(M)}$ denote the Grassmannian of $2$-planes in $M$%
, and let $\pi :\mathcal{G}r\mathcal{(M)\rightarrow M}$ be the projection of 
$\mathcal{G}r\mathcal{(M)}$ to $M$. As in \cite{MurphWilh}, we fix a metric
on $\mathcal{G}r\mathcal{(M)}$ so that $\pi $ is a Riemannian submersion
with totally geodesic fibers that are isometric to the Grassmannian of $2$
-planes in $\mathbb{R}^{3}$. For a metric space $X$, $A\subset X$, and $r>0$
, let 
\begin{equation*}
B_{r}(A)=\{x\in X:\operatorname{dist}(x,A)<r\}.
\end{equation*}

Let $\mathcal{R}_{0}$ be the set of $G$--rigid 2-planes for $g$.

We put a bar $\overline{}$ over a set to denote its closure. Thus $\bar{A}$
is the closure of $A.$

Throughout Section \ref{local construct section} we fix a coordinate
neighborhood $U$ that has the property that $\bar{U}$ has a neighborhood $V$
which is also a coordinate neighborhood. We use the notation $O\left(
\varepsilon \right) $ to denote a quantity whose absolute value is no larger
than $C\varepsilon ,$ where $C$ is a positive constant that only depends on $%
U$ and $V.$

%---------------------------------------------------------------------------------------------------------------------------------%

\section{The Local Construction\label{local construct section}}

In this section, we state and prove the main lemma of the paper, which gives
the method to perform our local deformation.

\begin{Main Lemma}
\label{local-main} %Main local lemma
Given $K,\eta >0$ and sufficiently small $\rho >0$, there are $\xi
,s_{0},\eta >0$ such that if 
\begin{equation*}
\left\vert g-\hat{g}\right\vert _{C^{3}}<\xi
\end{equation*}%
and $P$ is a $G$--rigid plane for $\hat{g},$ then there is a $C^{\infty }$
-family of metrics $\{g_{s}\}_{s\in \lbrack 0,s_{0}]}$ so that the following
hold.

\begin{enumerate}
\item $g_{0}=\hat{g}.$

\item For all $s$, $g_{s}=\hat{g}$ on $M\setminus B_{\rho +\eta }(\pi (P))$.

\item Let $\sigma (P)$ be the section of $\mathcal{G}r\mathcal{(}B_{\rho } 
\mathcal{(\pi (P)))}$ determined by $P$ via normal coordinates at $\pi (P)$
with respect to $g$. For all $s\in \left( 0,s_{0}\right) $ and $\check{P}\in
\pi ^{-1}(B_{\rho }(\pi (P)))\cap B_{\rho }(\sigma (P)),$ 
\begin{equation*}
G^{s}\left( \check{P}\right) >Ks,
\end{equation*}
where $G^{s}$ is the generic plane operator for $g_{s}$.
\end{enumerate}
\end{Main Lemma}

The proof begins with the following construction of a function on a
coordinate neighborhood.

\begin{lemma}
\label{local-M}There is a \ $K_{0}>1$ with the following property. Given $%
K>K_{0},$ $\varepsilon >0,$ $p\in M,$ and coordinate neighborhoods $U$ and $%
V $ of $p$ with $\overline{U}\subset V,$ there is a $C^{\infty }$ function $%
f:M\rightarrow \mathbb{R}$ so that

\begin{enumerate}
\item On $M$ 
\begin{equation*}
\left\vert f\right\vert _{C^{1}}<\varepsilon .
\end{equation*}

\item Unless $i=j=1,$ 
\begin{equation*}
\left\vert \partial _{j}\partial _{k}f\right\vert <\varepsilon .
\end{equation*}

\item For $i,j\in \{1,2,3\}$ and $k\in \{2,3\}$, on $U,$ 
\begin{equation*}
\partial _{i}\partial _{j}\partial _{k}f=0
\end{equation*}
and on $V$ 
\begin{equation}
\left\vert \partial _{1}\partial _{1}f\right\vert <K.  \label{pure 1 on V}
\end{equation}

\item $f\equiv 0$ on $M\setminus V$.

\item On $U,$ if 
\begin{equation*}
\left\vert \partial _{1}\partial _{1}f\right\vert \leq \sqrt{K},
\end{equation*}
then 
\begin{equation*}
\left\vert \partial _{1}\partial _{1}\partial _{1}f\right\vert \geq K.
\end{equation*}
In particular, on $U$ 
\begin{equation}
\max \left\{ \left\vert \partial _{1}\partial _{1}f\right\vert ,\left\vert
\partial _{1}\partial _{1}\partial _{1}f\right\vert \right\} >\sqrt{K}\text{
. }  \label{max on U}
\end{equation}
\end{enumerate}
\end{lemma}

The proof of Lemma \ref{local-M} relies on the following single variable
calculus result.

\begin{lemma}
\label{local-R} %constructs functions on R
There is a \ $K_{0}>1$ and for all $K>K_{0},$ and $\varepsilon \in (0,1)$,
there is a $C^{\infty }$ function $h:\mathbb{R}\rightarrow \mathbb{R}$ with
the following properties:

\begin{enumerate}
\item The $C^{1}$--norm of $h$ is bounded from above by $\varepsilon ,$ that
is, 
\begin{equation}
\left\vert h\right\vert _{C^{1}}<\varepsilon .  \label{local-R-eqn3}
\end{equation}

\item If 
\begin{equation*}
|h^{\prime \prime }(t)|\leq \sqrt{K},
\end{equation*}%
then 
\begin{equation*}
|h^{\prime \prime \prime }(t)|\geq K.
\end{equation*}%
In particular, 
\begin{equation}
\sqrt{K}<\max \{|h^{\prime \prime }(t)|,|h^{\prime \prime \prime }(t)|\}.
\end{equation}

\item For all $t,$ 
\begin{equation}
|h^{\prime \prime }(t)|\leq \frac{1}{2}K.  \label{h 2nd deriv upper bnd}
\end{equation}
\end{enumerate}
\end{lemma}

\begin{proof}
For $\eta =\frac{\varepsilon }{K}$, let 
\begin{equation*}
h\left( t\right) =\frac{1}{2}K\eta ^{2}\sin \left( \frac{t}{\eta }\right) .
\end{equation*}
Then 
\begin{eqnarray*}
\left\vert h\left( t\right) \right\vert &\leq &\frac{1}{2}K\eta ^{2}=\frac{1 
}{2}K\frac{\varepsilon ^{2}}{K^{2}}<\varepsilon ,\text{ and} \\
h^{\prime }\left( t\right) &=&\frac{1}{2}K\eta \cos \left( \frac{t}{\eta }
\right) ,\text{ so} \\
\left\vert h^{\prime }\left( t\right) \right\vert &\leq &\frac{1}{2}
\varepsilon ,\text{ and} \\
\left\vert h\right\vert _{C^{1}} &<&\varepsilon ,
\end{eqnarray*}
proving Part 1.

Also,

\begin{equation*}
h^{\prime \prime }\left( t\right) =-\frac{1}{2}K\sin \left( \frac{t}{\eta }
\right) .
\end{equation*}
So 
\begin{equation*}
\left\vert h^{\prime \prime }\left( t\right) \right\vert \leq \frac{1}{2}K,
\end{equation*}
proving Part 3.

We also have 
\begin{eqnarray*}
h^{\prime \prime \prime }\left( t\right) &=&-\frac{1}{2}\frac{K}{\eta }\cos
\left( \frac{t}{\eta }\right) \\
&=&-\frac{1}{2}\frac{K^{2}}{\varepsilon }\cos \left( \frac{t}{\eta }\right) .
\end{eqnarray*}%
Thus, if 
\begin{eqnarray*}
\left\vert h^{\prime \prime }\left( t\right) \right\vert &\leq &\sqrt{K},%
\text{ then} \\
\frac{1}{4}K^{2}\sin ^{2}\left( \frac{t}{\eta }\right) &\leq &K,\text{ so} \\
1-\cos ^{2}\left( \frac{t}{\eta }\right) &=&\sin ^{2}\left( \frac{t}{\eta }%
\right) \leq \frac{4}{K},
\end{eqnarray*}%
and 
\begin{eqnarray*}
\left\vert h^{\prime \prime \prime }\left( t\right) \right\vert ^{2} &=&%
\frac{1}{4}\frac{K^{4}}{\varepsilon ^{2}}\cos ^{2}\left( \frac{t}{\eta }%
\right) \\
&\geq &\frac{1}{4}\frac{K^{4}}{\varepsilon ^{2}}\left( 1-\frac{4}{K}\right)
\\
&>&K^{2},
\end{eqnarray*}%
if $K$ is sufficiently large and $\varepsilon $ is sufficiently small,
proving Part 2.
\end{proof}

\begin{proof}[Proof of Lemma \protect\ref{local-M}]
%Proof of constructing functions on M
Let $\Phi :V\longrightarrow \mathbb{R}^{3}$ be the coordinate chart. Let $%
\pi _{1}:\mathbb{R}^{3}\rightarrow \mathbb{R}$ be the orthogonal projection
onto the first factor. Let $\chi :\mathbb{R}^{3}\rightarrow \lbrack 0,1]$ be 
$C^{\infty }$ and satisfy 
\begin{align}
\chi |_{\Phi \left( U\right) }& \equiv 1,  \notag \\
\chi |_{\mathbb{R}^{3}\setminus \left\{ \Phi \left( V\right) \right\} }&
\equiv 0.  \label{local-M-eqn-compact-support}
\end{align}
Let $m_{3}>1$ satisfy 
\begin{equation}
\left\vert \chi \right\vert _{C^{3}}\leq m_{3}.  \label{local-M-eqn-C23bound}
\end{equation}
Choose a $C^{\infty }$ function $h:\mathbb{R}\rightarrow \mathbb{R}$ that
satisfies the conclusion of Lemma \ref{local-R} with $\varepsilon $ replaced
by $\frac{\varepsilon }{3m_{3}}.$ Thus 
\begin{equation}
\left\vert h\right\vert _{C^{1}}\leq \frac{\varepsilon }{3m_{3}},
\label{local-M-eqn-C1bound}
\end{equation}
\begin{equation}
\max \{|h^{\prime \prime }(t)|,|h^{\prime \prime \prime }(t)|\}>\sqrt{K},
\label{local-M-eqn-hprimebound}
\end{equation}
and 
\begin{equation}
|h^{\prime \prime }(t)|<\frac{1}{2}K,
\end{equation}
and 
\begin{equation*}
\text{if }|h^{\prime \prime }(t)|\leq \sqrt{K},\text{ then }|h^{\prime
\prime \prime }(t)|\geq 2K.
\end{equation*}

Let $\tilde{f}:\mathbb{R}^{3}\rightarrow \mathbb{R}$ be defined by 
\begin{equation}
\tilde{f}(p)=\chi (p)\cdot (h\circ \pi _{1})(p),  \label{local-M-eqn-ftilde}
\end{equation}
and let $f:M\rightarrow \mathbb{R}$ be 
\begin{equation*}
f=\tilde{f}\circ \Phi .
\end{equation*}
Since $\chi |_{\mathbb{R}^{3}\setminus \left\{ \Phi \left( V\right) \right\}
}\equiv 0,$ $f\equiv 0$ on $M\setminus V$, and Part 4 holds.

Noting that 
\begin{equation}
\partial _{k}\tilde{f}(p)=\partial _{k}\chi (p)\cdot \left( h\circ \pi
_{1}\right) (p)+\chi (p)\cdot \partial _{k}\left( h_{i}\circ \pi _{1}\right)
(p),  \label{1st order partials}
\end{equation}
$\left\vert h\right\vert _{C^{1}}\leq \frac{\varepsilon }{3m_{3}},$ and $%
\left\vert \chi \right\vert _{C^{3}}\leq m_{3},$ we conclude that 
\begin{equation*}
\left\vert f\right\vert _{C^{1}}<\varepsilon ,
\end{equation*}
so Part 1 holds.

From (\ref{1st order partials}) we get 
\begin{align}
\partial _{j}\partial _{k}\tilde{f}(p)& =\partial _{j}\partial _{k}\chi
(p)\cdot \left( h\circ \pi _{1}\right) (p)+\partial _{j}\chi (p)\cdot
\partial _{k}\left( h\circ \pi _{1}\right) (p)  \notag \\
& +\partial _{k}\chi (p)\cdot \partial _{j}\left( h\circ \pi _{1}\right)
(p)+\chi (p)\cdot \partial _{j}\partial _{k}\left( h\circ \pi _{1}\right) (p)
\label{2nd partials}
\end{align}

and 
\begin{align}
\partial _{i}\partial _{j}\partial _{k}\tilde{f}(p)& =\partial _{i}\partial
_{j}\partial _{k}\chi (p)\cdot \left( h\circ \pi _{1}\right) (p)+\partial
_{j}\partial _{k}\chi (p)\cdot \partial _{i}\left( h\circ \pi _{1}\right) (p)
\notag \\
& +\partial _{i}\partial _{j}\chi (p)\cdot \partial _{k}\left( h\circ \pi
_{1}\right) (p)+\partial _{j}\chi (p)\cdot \partial _{i}\partial _{k}\left(
h\circ \pi _{1}\right) (p)  \notag \\
& +\partial _{i}\partial _{k}\chi (p)\cdot \partial _{j}\left( h\circ \pi
_{1}\right) (p)+\partial _{k}\chi (p)\cdot \partial _{i}\partial _{j}\left(
h\circ \pi _{1}\right) (p)  \notag \\
& +\partial _{i}\chi (p)\cdot \partial _{j}\partial _{k}\left( h\circ \pi
_{1}\right) (p)+\chi (p)\cdot \partial _{i}\partial _{j}\partial _{k}\left(
h\circ \pi _{1}\right) (p).  \label{third partials}
\end{align}

In particular, since $\chi |_{\Phi \left( U\right) }\equiv 1,$ on $\Phi
\left( U\right) ,$ 
\begin{equation*}
\partial _{j}\partial _{k}\tilde{f}(p)=\partial _{j}\partial _{k}\left(
h\circ \pi _{1}\right) (p)\text{ and }\partial _{i}\partial _{j}\partial _{k}%
\tilde{f}(p)=\partial _{i}\partial _{j}\partial _{k}\left( h\circ \pi
_{1}\right) (p).
\end{equation*}%
Part 5 follows from this and Lemma \ref{local-R}. Since $\left\vert
h\right\vert _{C^{1}}\leq \frac{\varepsilon }{3m_{3}}$ and $\left\vert \chi
\right\vert _{C^{3}}\leq m_{3},$ we conclude from (\ref{2nd partials}) that 
\begin{equation*}
\left\vert \chi (p)\cdot \partial _{j}\partial _{k}\left( h\circ \pi
_{1}\right) (p)-\partial _{j}\partial _{k}\tilde{f}(p)\right\vert
<\varepsilon .
\end{equation*}%
Inequality (\ref{pure 1 on V}) follows from this and Inequality (\ref{h 2nd
deriv upper bnd}), provided $\varepsilon $ is sufficiently small and $K$ is
sufficiently large. Similarly, if either $i$ or $j$ is different from $1,$
then $\partial _{j}\partial _{k}\left( h\circ \pi _{1}\right) =0,$ so 
\begin{equation*}
\left\vert \partial _{j}\partial _{k}\tilde{f}(p)\right\vert <\varepsilon ,
\end{equation*}%
proving Part 2.
\end{proof}

%--------------------------------%
%		LOCAL PROOF SETUP		 %
%--------------------------------%

To continue the proof of the main lemma (\ref{local-main}) let $P\in 
\mathcal{R}_{0}$ and $\hat{g}$ be as in \ref{local-main} and have foot point 
$p$. Let $\{v,T,n\}$ be an ordered $\hat{g}$--orthonormal triplet at $p$
with 
\begin{equation}
{v,T}\in P\text{ and }n\text{ normal to }P.  \label{local-frame-text}
\end{equation}%
Let $\{E_{i}\}_{i=1}^{3}$ be an ordered coordinate frame that is defined on
a coordinate neighborhood 
\begin{equation}
U\text{ of }p\text{ with }E_{1}\left( p\right) =v,E_{2}\left( p\right) =T,%
\text{ and }E_{3}\left( p\right) =n.  \label{U dfn}
\end{equation}%
We further assume, as in Lemma \ref{local-M}, that the coordinate chart of $%
U $ is defined on a neighborhood $V$ of the closure of $U.$

Choose $g_{s}$ so that with respect to $\{E_{i}\}_{i=1}^{3}$, the matrix of $%
g_{s}-\hat{g}$ is 
\begin{equation}
\begin{pmatrix}
0 & 0 & 0 \\ 
0 & 0 & sf \\ 
0 & sf & 0%
\end{pmatrix}
\label{gs mibnus g hat}
\end{equation}%
where $f$ was constructed via Lemma \ref{local-M} with $U=B_{\rho }(\pi (P))$
.

We write $\tilde{g}$ for $g_{s}$ and use $\tilde{}$ for objects associated
with $\tilde{g}$. With respect to $\{E_{i}\}_{i=1}^{3}$, the Christoffel
symbols of the first kind are 
\begin{equation*}
\tilde{\Gamma}_{ij,k}=\tilde{g}\left( \tilde{\nabla}_{E_{i}}E_{j},E_{k}%
\right) .
\end{equation*}%
To emphasize the special role of our first coordinate vector field we will
write \textquotedblleft $v$\textquotedblright\ for the index
\textquotedblleft $1$\textquotedblright\ that corresponds to $E_{1}.$

Letting $\tilde{R}_{ijkl}:=\tilde{R}\left( E_{i},E_{j},E_{k},E_{l}\right) ,$
we then have 
\begin{equation}
\tilde{R}_{ijkl}=\partial _{i}\tilde{\Gamma}_{jk,l}-\partial _{j}\tilde{%
\Gamma}_{ik,l}+\tilde{g}^{\sigma \tau }\left( \tilde{\Gamma}_{ik,\sigma }%
\tilde{\Gamma}_{jl,\tau }-\tilde{\Gamma}_{jk,\sigma }\tilde{\Gamma}_{il,\tau
}\right) ,  \label{curvature}
\end{equation}%
where the Einstein summation convention is used on indices $\sigma $ and $%
\tau $ (see, for example, page 89 of \cite{Pet}).

We calculate $\tilde{\nabla}\tilde{R}$ in terms of Christoffel symbols of
the second kind which are 
\begin{equation}
\tilde{\nabla}_{E_{i}}E_{j}=\Gamma _{ij}^{k}E_{k},  \label{first kind}
\end{equation}
where the Einstein summation convention is used. We then have that $\left( 
\tilde{\nabla}_{v}\tilde{R}\right) \left( E_{2},v,v,E_{3}\right) $ is

\begin{eqnarray*}
\left( \tilde{\nabla}_{v}\tilde{R}\right) \left( E_{2},v,v,E_{3}\right) &:&= 
\tilde{\nabla}_{v}\left( \tilde{R}\left( E_{2},v,v,E_{3}\right) \right) - 
\tilde{R}\left( \tilde{\nabla}_{v}E_{2},v,v,E_{3}\right) \\
&&-\tilde{R}\left( E_{2},\tilde{\nabla}_{v}v,v,E_{3}\right) -\tilde{R}\left(
E_{2},v,\tilde{\nabla}_{v}v,E_{3}\right) \\
&&-\tilde{R}\left( E_{2},v,v,\tilde{\nabla}_{v}E_{3}\right) .
\end{eqnarray*}
Together with (\ref{first kind}) this gives us

\begin{eqnarray}
2\left( \tilde{\nabla}_{v}\tilde{R}\right) _{2vv3} &=&2\left( \partial _{v} 
\tilde{R}_{2vv3}-\tilde{\Gamma}_{v2}^{\tau }\tilde{R}_{\tau vv3}-\tilde{
\Gamma}_{vv}^{\tau }\tilde{R}_{2\tau v3}-\tilde{\Gamma}_{vv}^{\tau }\tilde{R}
_{2v\tau 3}-\tilde{\Gamma}_{v3}^{\tau }\tilde{R}_{2vv\tau }\right)
\label{cov-curvature2} \\
&=&2\partial _{v}\tilde{R}_{2vv3}-\tilde{g}^{\tau \mu }\tilde{\Gamma}
_{v2,\mu }\tilde{R}_{\tau vv3}-\tilde{g}^{\tau \mu }\tilde{\Gamma}_{vv,\mu } 
\tilde{R}_{2\tau v3}-\tilde{g}^{\tau \mu }\tilde{\Gamma}_{vv,\mu }\tilde{R}
_{2v\tau 3}-\tilde{g}^{\tau \mu }\tilde{\Gamma}_{v3,\mu }\tilde{R}_{2vv\tau
},  \notag
\end{eqnarray}
where we use the Einstein convention and the formula 
\begin{equation*}
\Gamma _{ij}^{k}=\frac{1}{2}g^{kl}\Gamma _{ij,l}
\end{equation*}
on the bottom of page 66 of \cite{Pet} to convert between the two
Christoffel types.

Before stating the next three results we remind the reader that we have
suppressed the role of $s$ by writing $\tilde{g}$ for $g_{s}$ and use $%
\tilde{}$ for objects associated to $\tilde{g}$. For example, this is
important to keep in mind when reading Part 3 of the following result.

\begin{proposition}
\label{setup-prop} %Proposition that directly leads into main local proof
Let $U$ be a coordinate neighborhood as in (\ref{U dfn}), let $i,j,k,l$ be
arbitrary elements of $\{1,2,3\}$, and let $\varepsilon >0$ be as in Lemma %
\ref{local-M}.

\begin{enumerate}
\item Writing $(\tilde{\Gamma}-\hat{\Gamma})_{ij,k}$ for $\tilde{\Gamma}
_{ij,k}-\hat{\Gamma}_{ij,k}$, we have 
\begin{equation}
\left( \left\vert \tilde{\Gamma}-\hat{\Gamma}\right\vert \right) _{ij,k}\leq
O(\varepsilon s).  \label{gamma-ineq}
\end{equation}

\item We have 
\begin{eqnarray}
\partial _{v}(\tilde{\Gamma}-\hat{\Gamma})_{2v,3} &=&\partial _{v}(\tilde{%
\Gamma}-\hat{\Gamma})_{v3,2}=-\partial _{v}(\tilde{\Gamma}-\hat{\Gamma}%
)_{23,v}=\partial _{v}(\tilde{\Gamma}-\hat{\Gamma})_{v2,3}  \notag \\
&=&\partial _{v}(\tilde{\Gamma}-\hat{\Gamma})_{3v,2}=-\partial _{v}(\tilde{%
\Gamma}-\hat{\Gamma})_{32,v}  \notag \\
&=&\frac{s}{2}\partial _{v}\partial _{v}f,  \label{big Christ deriv}
\end{eqnarray}%
and all of the other expressions 
\begin{equation*}
\partial _{i}(\tilde{\Gamma}-\hat{\Gamma})_{jk,l}
\end{equation*}%
with $i,j,k\in \left\{ 1,2,3\right\} $ have absolute value $<O(\varepsilon
s) $.

\item There is a $C>0$ so that for all $s\in \left[ 0,1\right] ,$ 
\begin{equation*}
\left\vert \tilde{\Gamma}_{ij,k}\right\vert \leq C\text{ and }\left\vert 
\hat{\Gamma}_{ij,k}\right\vert \leq C
\end{equation*}%
throughout $U.$

\item We have 
\begin{eqnarray}
\partial _{v}\partial _{v}(\tilde{\Gamma}-\hat{\Gamma})_{2v,3} &=&\partial
_{v}\partial _{v}(\tilde{\Gamma}-\hat{\Gamma})_{v3,2}=-\partial _{v}\partial
_{v}(\tilde{\Gamma}-\hat{\Gamma})_{23,v}=\partial _{v}\partial _{v}(\tilde{
\Gamma}-\hat{\Gamma})_{v2,3}  \notag \\
&=&\partial _{v}\partial _{v}(\tilde{\Gamma}-\hat{\Gamma})_{3v,2}=-\partial
_{v}\partial _{v}(\tilde{\Gamma}-\hat{\Gamma})_{32,v}  \notag \\
&=&\frac{s}{2}\partial _{v}\partial _{v}\partial _{v}f.
\label{big 2nd deriv}
\end{eqnarray}

Thus by Part 5 of Proposition \ref{local-M}, if $\partial _{v}(\tilde{\Gamma}
-\hat{\Gamma})_{2v,3}\leq 2\sqrt{K}s$, then $\partial _{v}\partial _{v}( 
\tilde{\Gamma}-\hat{\Gamma})_{2v,3}\geq 2Ks$ on $U$.
\end{enumerate}
\end{proposition}

\begin{proof}
The proofs of Parts 1 and 2 are identical to the proofs of Propositions 2.5
and 2.6 in \cite{MurphWilh}.

To prove Part 3, first recall that our coordinates and hence our Christoffel
symbols for all of our metrics are defined on a neighborhood $V$ of closure
of $U.$ Part 3 follows from this and the fact that $\overline{U}\times \left[
0,1\right] $ is compact.

To prove Part 4, we note that by the Koszul formula, each of the expressions
is equal to 
\begin{align*}
& \frac{1}{2}\partial _{v}\partial _{v}\left[ \partial _{v}(\tilde{g}-\hat{g}%
)(E_{2},E_{3})+\partial _{2}(\tilde{g}-\hat{g})(E_{3},E_{v})-\partial _{3}(%
\tilde{g}-\hat{g})(E_{v},E_{2})\right] \\
& =\frac{1}{2}\partial _{v}\partial _{v}\partial _{v}(\tilde{g}-\hat{g}%
)(E_{2},E_{3})\text{ by (\ref{gs mibnus g hat})} \\
& =\frac{s}{2}\partial _{v}\partial _{v}\partial _{v}(f).
\end{align*}
\end{proof}

Additionally we need

\begin{proposition}
\label{inverse difference prop}Let $U$ be a coordinate neighborhood as in ( %
\ref{U dfn}). There is a $C>0$ so that on $U$ the coefficients $\tilde{g}%
^{ij}$ and $\hat{g}^{ij}$ of the inverses of $\{\tilde{g}\}_{ij}$ and $\{%
\hat{g}\}_{ij}$ satisfy 
\begin{equation}
\left\vert \tilde{g}^{ij}-\hat{g}^{ij}\right\vert _{C^{1}}<C(\varepsilon s),
\label{C1 bnd on inverse coeffs}
\end{equation}%
and for all $s\in \left[ 0,1\right] ,$ 
\begin{equation}
\max \left\{ \tilde{g}^{ij},\hat{g}^{ij}\right\} <C
\label{bnd on inverse
				matrix}
\end{equation}%
throughout $U.$
\end{proposition}

\begin{proof}
It follows from Lemma \ref{local-M} and the definition of $\left( \tilde{g}-%
\hat{g}\right) _{ij}$ in (\ref{gs mibnus g hat}) that 
\begin{equation}
\left\vert \left( \tilde{g}-\hat{g}\right) _{ij}\right\vert _{C^{1}}<O\left(
\varepsilon s\right) .  \label{C1 close metric coeffis}
\end{equation}%
Since $\tilde{g}_{ij}\ $and $\hat{g}_{ij}$ are defined on a neighborhood $V$
of $\overline{U}$ they define maps 
\begin{eqnarray}
\hat{G} &:&\overline{U}\longrightarrow Gl\left( n\right) \text{ and }\tilde{G%
}:\overline{U}\times \left[ -\varepsilon ,\varepsilon \right]
\longrightarrow Gl\left( n\right)  \label{finite cover} \\
\hat{G}\left( u\right) &=&\left( \hat{g}_{ij}\left( u\right) \right) _{ij}%
\text{ and }\tilde{G}\left( u,s\right) =\left( \tilde{g}_{ij}^{s}\left(
u\right) \right) _{ij}  \notag
\end{eqnarray}%
that take values in a compact subset of $Gl\left( n\right) .$ This, together
with (\ref{C1 close metric coeffis}) and the fact that the inversion map $%
\mathcal{I}:Gl\left( n\right) \longrightarrow Gl\left( n\right) $ is $%
C^{\infty },$ gives us that there is a $C>0$ so that 
\begin{equation*}
\left\vert \tilde{g}^{ij}-\hat{g}^{ij}\right\vert _{C^{1}}<C(\varepsilon s)
\end{equation*}%
as claimed in (\ref{C1 bnd on inverse coeffs}).

The proof of (\ref{bnd on inverse matrix}) is similar.
\end{proof}

We continue to write $v$ for the first element in our frame.

\begin{proposition}
\label{setup-prop2} Let $i,j,k,l$ be arbitrary elements of $\{1,2,3\}$.

\begin{enumerate}
\item We have 
\begin{equation}
\left\vert \left( \tilde{R}-\hat{R}\right) _{ijkl}\right\vert \leq
O(\varepsilon s),  \label{curvature-ineq}
\end{equation}%
except for up to a symmetry of $\tilde{R}-\hat{R},$ the case of $(\tilde{R}-%
\hat{R})_{2vv3}.$ In that event, we have 
\begin{equation}
(\tilde{R}-\hat{R})_{2vv3}=\partial _{v}\left( \hat{\Gamma}-\tilde{\Gamma}%
\right) _{2v,3}+O(\varepsilon s)=-\frac{s}{2}\partial _{v}\partial
_{v}f+O(\varepsilon s).  \label{curvature-ineq2}
\end{equation}

\item There is a $C>0$ so that for all $s\in \left[ 0,1\right] ,$ 
\begin{equation*}
\left\vert \tilde{R}_{ijkl}\right\vert \leq C\text{ and }\left\vert \hat{R}%
_{ijkl}\right\vert \leq C
\end{equation*}%
throughout $U.$

\item For $K$ sufficiently large and as in Lemma \ref{local-M}, if $%
\left\vert (\tilde{R}-\hat{R})_{2vv3}\right\vert \leq s\sqrt{K},$ then 
\begin{equation}
(\tilde{\nabla}_{v}\tilde{R}-\hat{\nabla}_{v}\hat{R})_{2vv3}\geq K.
\end{equation}%
In particular, 
\begin{equation*}
\max \left\{ \left\vert (\tilde{R}-\hat{R})_{2vv3}\right\vert ,\left\vert (%
\tilde{\nabla}_{v}\tilde{R}-\hat{\nabla}_{v}\hat{R})_{2vv3}\right\vert
\right\} \geq s\sqrt{K}.
\end{equation*}
\end{enumerate}
\end{proposition}

The proof uses two calculus estimates that we make explicit in the next two
lemmas.

\begin{lemma}
\label{C0 bnd Lemma}Let $f,g,h,\tilde{f},\tilde{g},\tilde{h}:\mathbb{R}%
\longrightarrow \mathbb{R}$ be $C^{1}$--functions. Suppose that all six
functions are bounded by $C$ and all three differences satisfy 
\begin{eqnarray*}
\left\vert f-\tilde{f}\right\vert &<&D \\
\left\vert g-\tilde{g}\right\vert &<&D \\
\left\vert h-\tilde{h}\right\vert &<&D,
\end{eqnarray*}%
then 
\begin{equation*}
\left\vert fgh-\tilde{f}\tilde{g}\tilde{h}\right\vert \leq 3DC^{2}
\end{equation*}
\end{lemma}

\begin{proof}
Notice that 
\begin{equation*}
fgh=(f-\tilde{f})gh+\tilde{f}(g-\tilde{g})h+\tilde{f}\tilde{g}(h-\tilde{h})+ 
\tilde{f}\tilde{g}\tilde{h}.
\end{equation*}
Due to the bounds on the values of $f,g,h,\tilde{f},\tilde{g},\tilde{h}$ and 
$(f-\tilde{f})$, $(g-\tilde{g})$, $(h-\tilde{h})$ from our hypothesis, we
conclude that 
\begin{equation*}
\left\vert fgh-\tilde{f}\tilde{g}\tilde{h}\right\vert <3DC^{2}.
\end{equation*}
\end{proof}

\begin{lemma}
\label{C1 bnd lemma}Suppose that all six functions are $C^{1}$--bounded by $%
C,$ and 
\begin{eqnarray*}
\left\vert f-\tilde{f}\right\vert _{C^{1}} &<&D, \\
\left\vert g-\tilde{g}\right\vert _{C^{1}} &<&D,\text{ and} \\
\left\vert h-\tilde{h}\right\vert _{C^{1}} &<&D,
\end{eqnarray*}%
then 
\begin{equation*}
\left\vert fgh-\tilde{f}\tilde{g}\tilde{h}\right\vert _{C^{1}}<9DC^{2}
\end{equation*}
\end{lemma}

\begin{proof}
We have 
\begin{align*}
(fgh)^{\prime }-(\tilde{f}\tilde{g}\tilde{h})^{\prime }& =(f^{\prime }gh- 
\tilde{f}^{\prime }\tilde{g}\tilde{h}) \\
& +(fg^{\prime }h-\tilde{f}\tilde{g}^{\prime }\tilde{h}) \\
& +(fgh^{\prime }-\tilde{f}\tilde{g}\tilde{h}^{\prime }).
\end{align*}
We apply the following expansion to the term $f^{\prime }gh$ located in the
right hand side of the previous display: 
\begin{equation*}
f^{\prime }gh=(f-\tilde{f})^{\prime }gh+\tilde{f}^{\prime }(g-\tilde{g})h+ 
\tilde{f}^{\prime }\tilde{g}(h-\tilde{h})+\tilde{f}^{\prime }\tilde{g}\tilde{
h}.
\end{equation*}
Doing the analogous substitution for $fg^{\prime }h$ and $fgh^{\prime }$
yields 
\begin{align*}
(fgh)^{\prime }-(\tilde{f}\tilde{g}\tilde{h})^{\prime }& =(f-\tilde{f}
)^{\prime }gh+\tilde{f}^{\prime }(g-\tilde{g})h+\tilde{f}^{\prime }\tilde{g}
(h-\tilde{h}) \\
& +(f-\tilde{f})g^{\prime }h+\tilde{f}(g-\tilde{g})^{\prime }h+\tilde{f} 
\tilde{g}^{\prime }(h-\tilde{h}) \\
& +(f-\tilde{f})gh^{\prime }+\tilde{f}(g-\tilde{g})h^{\prime }+\tilde{f} 
\tilde{g}(h-\tilde{h})^{\prime }.
\end{align*}
Combining this with our hypothesis, we get 
\begin{equation*}
\left\vert fgh-\tilde{f}\tilde{g}\tilde{h}\right\vert _{C^{1}}<9DC^{2}.
\end{equation*}
\end{proof}

\begin{proof}[Proof of Proposition \protect\ref{setup-prop2}]
By Equation \ref{curvature}, we have 
\begin{eqnarray}
\tilde{R}_{ijkl} &=&\partial _{i}\tilde{\Gamma}_{jk,l}-\partial _{j}\tilde{%
\Gamma}_{ik,l}+\tilde{g}^{\sigma \tau }\left( \tilde{\Gamma}_{ik,\sigma }%
\tilde{\Gamma}_{jl,\tau }-\tilde{\Gamma}_{jk,\sigma }\tilde{\Gamma}_{il,\tau
}\right) \text{ and}  \label{R tilde} \\
\hat{R}_{ijkl} &=&\partial _{i}\hat{\Gamma}_{jk,l}-\partial _{j}\hat{\Gamma}%
_{ik,l}+\left( \hat{g}\right) ^{\sigma \tau }\left( \hat{\Gamma}_{ik,\sigma }%
\hat{\Gamma}_{jl,\tau }-\hat{\Gamma}_{jk,\sigma }\hat{\Gamma}_{il,\tau
}\right) .  \label{R hat}
\end{eqnarray}

If $i=j=v,$ then $\tilde{R}_{ijkl}=\hat{R}_{ijkl}=0.$ Otherwise, at most one
of $\partial _{i}\left( \tilde{\Gamma}-\hat{\Gamma}\right) _{jk,l}$ or $%
\partial _{j}\left( \tilde{\Gamma}-\hat{\Gamma}\right) _{ik,l}$ can
correspond to the indices in (\ref{big Christ deriv}). In the event it is $%
\partial _{i}\left( \tilde{\Gamma}-\hat{\Gamma}\right) _{jk,l}$ , then $i=v$
and $j,k,l$ are distinct elements of $\left\{ v,2,3\right\} ,$ so the
curvature difference $\left( \tilde{R}-\hat{R}\right) _{ijkl}$ is up to a
sign the exceptional one in (\ref{curvature-ineq2}). For similar reasons, if 
$\partial _{j}\left( \tilde{\Gamma}-\hat{\Gamma}\right) _{ik,l}$ corresponds
to the indices in (\ref{big Christ deriv}), then curvature difference $%
\left( \tilde{R}-\hat{R}\right) _{ijkl}$ is up to a sign the exceptional one
in (\ref{curvature-ineq2}). If neither $\partial _{i}\left( \tilde{\Gamma}-%
\hat{\Gamma}\right) _{jk,l}$ nor $\partial _{j}\left( \tilde{\Gamma}-\hat{%
\Gamma}\right) _{ik,l}$ corresponds to the indices in (\ref{big Christ deriv}%
), then by combining Propositions \ref{setup-prop} and \ref{inverse
difference prop} with Equations \ref{R tilde} and \ref{R hat} and the two
calculus lemmas above, we conclude that 
\begin{equation*}
\left\vert \left( \tilde{R}-\hat{R}\right) _{ijkl}\right\vert \leq
O(\varepsilon s),\text{ proving }(\ref{curvature-ineq}).
\end{equation*}

To prove (\ref{curvature-ineq2}), note that in this special case, Equations %
\ref{R tilde} and \ref{R hat} become 
\begin{equation*}
\tilde{R}_{2vv3}=\partial _{2}\tilde{\Gamma}_{vv,3}-\partial _{v}\tilde{%
\Gamma}_{2v,3}+\tilde{g}^{\sigma \tau }\left( \tilde{\Gamma}_{2v,\sigma }%
\tilde{\Gamma}_{v3,\tau }-\tilde{\Gamma}_{vv,\sigma }\tilde{\Gamma}_{23,\tau
}\right)
\end{equation*}%
and

\begin{equation*}
\hat{R}_{2vv3}=\partial _{2}\hat{\Gamma}_{vv,3}-\partial _{v}\hat{\Gamma}%
_{2v,3}+\tilde{g}^{\sigma \tau }\left( \hat{\Gamma}_{2v,\sigma }\hat{\Gamma}%
_{v3,\tau }-\hat{\Gamma}_{vv,\sigma }\hat{\Gamma}_{23,\tau }\right)
\end{equation*}%
It follows from the Koszul formula that 
\begin{eqnarray*}
2\tilde{\Gamma}_{vv,3} &=&2\tilde{g}\left( \tilde{\nabla}_{E_{V}}E_{V},E_{3}%
\right) \\
&=&2D_{E_{V}}\tilde{g}\left( E_{V},E_{3}\right) -D_{E_{3}}\tilde{g}\left(
E_{V},E_{V}\right)
\end{eqnarray*}%
and similarly%
\begin{equation*}
2\hat{\Gamma}_{vv,3}=2D_{E_{V}}\hat{g}\left( E_{V},E_{3}\right) -D_{E_{3}}%
\hat{g}\left( E_{V},E_{V}\right) .
\end{equation*}%
However, from the definition of $\tilde{g}-\hat{g},$ in (\ref{gs mibnus g
hat}) we see that 
\begin{equation*}
\tilde{g}\left( E_{V},E_{3}\right) -\hat{g}\left( E_{V},E_{3}\right) =\tilde{%
g}\left( E_{V},E_{V}\right) -\hat{g}\left( E_{V},E_{V}\right) \equiv 0.
\end{equation*}%
Hence 
\begin{equation}
\partial _{2}\tilde{\Gamma}_{vv,3}-\partial _{2}\hat{\Gamma}_{vv,3}\equiv 0.
\label{special Christofell}
\end{equation}%
Thus 
\begin{equation*}
\left( \tilde{R}-\hat{R}\right) _{2vv3}-\partial _{v}\left( \hat{\Gamma}-%
\tilde{\Gamma}\right) _{2v,3}=\tilde{g}^{\sigma \tau }\left( \tilde{\Gamma}%
_{2v,\sigma }\tilde{\Gamma}_{v3,\tau }-\tilde{\Gamma}_{vv,\sigma }\tilde{%
\Gamma}_{23,\tau }\right) -\tilde{g}^{\sigma \tau }\left( \hat{\Gamma}%
_{2v,\sigma }\hat{\Gamma}_{v3,\tau }-\hat{\Gamma}_{vv,\sigma }\hat{\Gamma}%
_{23,\tau }\right) .
\end{equation*}

By combining Propositions \ref{setup-prop} and \ref{inverse difference prop}
with Lemma \ref{C0 bnd Lemma}, we see that the right hand side of the
previous display is $\leq O\left( \varepsilon s\right) .$ Thus 
\begin{equation*}
\left\vert (\tilde{R}-\hat{R})_{2vv3}-\partial _{v}\left( \hat{\Gamma}-%
\tilde{\Gamma}\right) _{2v,3}\right\vert \leq O\left( \varepsilon s\right) ,
\end{equation*}%
and by Part 2 of Proposition \ref{setup-prop}, $\partial _{v}\left( \hat{%
\Gamma}-\tilde{\Gamma}\right) _{2v,3}=-\frac{s}{2}\partial _{v}\partial
_{v}f,$ so 
\begin{equation*}
(\tilde{R}-\hat{R})_{2vv3}=\partial _{v}\left( \hat{\Gamma}-\tilde{\Gamma}%
\right) _{2v,3}+O(\varepsilon s)=-\frac{s}{2}\partial _{v}\partial
_{v}f+O(\varepsilon s),
\end{equation*}%
as claimed.

Part 2 follows from a compactness argument as in the proof of Part 3 of
Proposition \ref{setup-prop}.

To prove Part 3, we recall that via Equation (\ref{cov-curvature2}) we have 
\begin{equation*}
2\left( \tilde{\nabla}_{v}\tilde{R}\right) _{2vv3}=2\partial _{v}\tilde{R}
_{2vv3}-\tilde{g}^{\tau \mu }\tilde{\Gamma}_{v2,\mu }\tilde{R}_{\tau vv3}- 
\tilde{g}^{\tau \mu }\tilde{\Gamma}_{vv,\mu }\tilde{R}_{2\tau v3}-\tilde{g}
^{\tau \mu }\tilde{\Gamma}_{vv,\mu }\tilde{R}_{2v\tau 3}-\tilde{g}^{\tau \mu
}\tilde{\Gamma}_{v3,\mu }\tilde{R}_{2vv\tau }
\end{equation*}
and 
\begin{equation*}
2\left( \hat{\nabla}_{v}\hat{R}\right) _{2vv3}=2\partial _{v}\hat{R}_{2vv3}- 
\hat{g}^{\tau \mu }\hat{\Gamma}_{v2,\mu }\hat{R}_{\tau vv3}-\hat{g}^{\tau
\mu }\hat{\Gamma}_{vv,\mu }\hat{R}_{2\tau v3}-\hat{g}^{\tau \mu }\hat{\Gamma}
_{vv,\mu }\hat{R}_{2v\tau 3}-\hat{g}^{\tau \mu }\hat{\Gamma}_{v3,\mu }\hat{R}
_{2vv\tau }.
\end{equation*}
Thus 
\begin{eqnarray}
&&2\left( \left( \tilde{\nabla}_{v}\tilde{R}\right) _{2vv3}-\partial _{v} 
\tilde{R}_{2vv3}\right) -2\left( \left( \hat{\nabla}_{v}\hat{R}\right)
_{2vv3}-2\partial _{v}\hat{R}_{2vv3}\right)  \notag \\
&=&\left( \hat{g}^{\tau \mu }\hat{\Gamma}_{v2,\mu }\hat{R}_{\tau vv3}-\tilde{
g}^{\tau \mu }\tilde{\Gamma}_{v2,\mu }\tilde{R}_{\tau vv3}\right) +\left( 
\hat{g}^{\tau \mu }\hat{\Gamma}_{vv,\mu }\hat{R}_{2\tau v3}-\hat{g}^{\tau
\mu }\hat{\Gamma}_{vv,\mu }\hat{R}_{2\tau v3}\right)  \notag \\
&&+\left( \hat{g}^{\tau \mu }\hat{\Gamma}_{vv,\mu }\hat{R}_{2v\tau 3}-\tilde{
g}^{\tau \mu }\tilde{\Gamma}_{vv,\mu }\tilde{R}_{2v\tau 3}\right) +\left( 
\hat{g}^{\tau \mu }\hat{\Gamma}_{v3,\mu }\hat{R}_{2vv\tau }-\tilde{g}^{\tau
\mu }\tilde{\Gamma}_{v3,\mu }\tilde{R}_{2vv\tau }\right) .
\label{covar diff}
\end{eqnarray}

By combining our hypothesis that $\left\vert (\tilde{R}-\hat{R}%
)_{2vv3}\right\vert \leq s\sqrt{K}$ with Parts 1 and 2 of this result,
Proposition \ref{setup-prop}, Proposition \ref{inverse difference prop}, and
Lemma \ref{C0 bnd Lemma}, we see that each of the four terms in the
parentheses on the right hand side of (\ref{covar diff}) is $\leq O\left( s%
\sqrt{K}\right) .$ So 
\begin{equation}
\left\vert 2\left( \left( \tilde{\nabla}_{v}\tilde{R}\right)
_{2vv3}-\partial _{v}\tilde{R}_{2vv3}\right) -2\left( \left( \hat{\nabla}_{v}%
\hat{R}\right) _{2vv3}-2\partial _{v}\hat{R}_{2vv3}\right) \right\vert \leq
O\left( s\sqrt{K}\right) .  \label{covar almot there}
\end{equation}

By Equation \ref{curvature}, we have 
\begin{equation*}
\partial _{v}\tilde{R}_{2vv3}=\partial _{v}\partial _{2}\tilde{\Gamma}%
_{vv,3}-\partial _{v}\partial _{v}\tilde{\Gamma}_{2v,3}+\partial _{v}\left( 
\tilde{g}^{\sigma \tau }\left( \tilde{\Gamma}_{2v,\sigma }\tilde{\Gamma}%
_{v3,\tau }-\tilde{\Gamma}_{vv,\sigma }\tilde{\Gamma}_{23,\tau }\right)
\right)
\end{equation*}%
and 
\begin{equation*}
\partial _{v}\hat{R}_{2vv3}=\partial _{v}\partial _{2}\hat{\Gamma}%
_{vv,3}-\partial _{v}\partial _{v}\hat{\Gamma}_{2v,3}+\partial _{v}\left( 
\hat{g}^{\sigma \tau }\left( \hat{\Gamma}_{2v,\sigma }\hat{\Gamma}_{v3,\tau
}-\hat{\Gamma}_{vv,\sigma }\hat{\Gamma}_{23,\tau }\right) \right) .
\end{equation*}

It follows from Equation (\ref{special Christofell}) that 
\begin{equation*}
\partial _{v}\partial _{2}\tilde{\Gamma}_{vv,3}-\partial _{v}\partial _{2}%
\hat{\Gamma}_{vv,3}\equiv 0.
\end{equation*}%
Thus 
\begin{eqnarray}
&&\left\vert \left( \partial _{v}\left( \tilde{R}_{2vv3}-\hat{R}%
_{2vv3}\right) \right) -\partial _{v}\partial _{v}\left( \hat{\Gamma}_{2v,3}-%
\tilde{\Gamma}_{2v,3}\right) \right\vert  \notag \\
&=&\partial _{v}\left( \left( \tilde{g}^{\sigma \tau }\tilde{\Gamma}%
_{2v,\sigma }\tilde{\Gamma}_{v3,\tau }-\hat{g}^{\sigma \tau }\hat{\Gamma}%
_{2v,\sigma }\hat{\Gamma}_{v3,\tau }\right) \right) +\partial _{v}\left( 
\hat{g}^{\sigma \tau }\hat{\Gamma}_{vv,\sigma }\hat{\Gamma}_{23,\tau }-%
\tilde{g}^{\sigma \tau }\tilde{\Gamma}_{vv,\sigma }\tilde{\Gamma}_{23,\tau
}\right) .  \label{almost there}
\end{eqnarray}

Equation (\ref{curvature-ineq2}), together with our hypothesis that $%
\left\vert (\tilde{R}-\hat{R})_{2vv3}\right\vert \leq s\sqrt{K}$, and Part 2
of Proposition \ref{setup-prop} together give us that 
\begin{equation}
\left\vert \partial _{v}\left( \tilde{\Gamma}-\hat{\Gamma}\right)
_{ij,k}\right\vert \leq 2s\sqrt{K}.  \label{Christ small}
\end{equation}

Combining this with Proposition \ref{setup-prop}, Proposition \ref{inverse
difference prop}, and Lemma \ref{C1 bnd lemma}, we see that each term on the
right hand side of (\ref{almost there}) is $\leq O\left( s\sqrt{K}\right) .$
So 
\begin{equation*}
\left\vert \left( \partial _{v}\tilde{R}_{2vv3}-\partial _{v}\hat{R}%
_{2vv3}\right) -\left( \partial _{v}\partial _{v}\left( \hat{\Gamma}_{2v,3}-%
\tilde{\Gamma}_{2v,3}\right) \right) \right\vert \leq O\left( s\sqrt{K}%
\right) .
\end{equation*}

From Part 4 of Proposition \ref{setup-prop}, we have $\partial _{v}\partial
_{v}\left( \hat{\Gamma}_{2v,3}-\tilde{\Gamma}_{2v,3}\right) =\frac{s}{2}%
\partial _{v}\partial _{v}\partial _{v}f.$ This, the previous display, and (%
\ref{covar almot there}) combine to give us that 
\begin{equation*}
\left\vert \left( \tilde{\nabla}_{v}\tilde{R}\right) _{2vv3}-\left( \hat{%
\nabla}_{v}\hat{R}\right) _{2vv3}-\frac{s}{2}\partial _{v}\partial
_{v}\partial _{v}f\right\vert \leq O\left( s\sqrt{K}\right) .
\end{equation*}

Inequality (\ref{Christ small}) together with (\ref{big Christ deriv}) gives
us $\left\vert \partial _{v}\partial _{v}f\right\vert \leq 4\sqrt{K}.$ Thus
by Part 5 of Lemma \ref{local-M}, with $4\sqrt{K}$ playing the role of $%
\sqrt{K}$, $\left\vert \partial _{v}\partial _{v}\partial _{v}f\right\vert
\geq 16K$ . Therefore 
\begin{equation*}
\left\vert \left( \tilde{\nabla}_{v}\tilde{R}\right) _{2vv3}-\left( \hat{%
\nabla}_{v}\hat{R}\right) _{2vv3}\right\vert \geq 16K-O\left( s\sqrt{K}%
\right) >K,
\end{equation*}%
if $K$ is sufficiently large.
\end{proof}

Part 1 of the main lemma (\ref{local-main}) follows from our definition of $%
g_{s}-\hat{g}$ in (\ref{gs mibnus g hat}). Part 2 of the main lemma follows
from our choice of $f$; see Part 4 of Lemma \ref{local-M}. Part 3 of the
main lemma follows from Part 3 of Proposition \ref{setup-prop2} and the
definition of the generic plane operator.

\section{The Global Argument\label{Global section}}

We prove the Density Assertion by repeatedly applying the main lemma (\ref%
{local-main}) and Lemma \ref{lem-C}. Except for changes of notation and
terminology, the argument is the same as the proof of the \textquotedblleft
Modified $l^{th}$--Partially Geodesic Assertion\textquotedblright\ in \cite%
{MurphWilh}, so we only give a brief description here.

First construct an open cover $\left\{ \mathcal{U}_{i}\right\} _{i=1}^{K}$
of the set $\mathcal{\ R}_{0}$ of $G$--rigid planes. We will apply the main
lemma and Lemma \ref{lem-C} successively to $\mathcal{U}_{1},$ $\mathcal{U}%
_{2},\ldots ,\mathcal{U}_{K}.$ So we assume that each $\mathcal{U}_{i}$ has
the form $B_{\rho }\left( \pi \left( P\right) \right) ,$ where the notation
is as in the main lemma. Prior to the first step, notice that since $\left\{ 
\mathcal{U}_{i}\right\} _{i=1}^{K}$ covers $\mathcal{R}_{0},$ the generic
plane operator $G:\mathcal{\ G}r\left( M\right) \longrightarrow \mathbb{R}$
is positive on $\mathcal{G}r\left( M\right) \setminus \cup _{i=1}^{K}%
\mathcal{U}_{i}.$ Since $\mathcal{G}r\left( M\right) \setminus \cup
_{i=1}^{K}\mathcal{U}_{i}$ is compact, there is a $\delta _{0}>0$ so that 
\begin{equation}
G|_{\mathcal{G}r\left( M\right) \setminus \cup _{i=1}^{K}\mathcal{U}%
_{i}}>\delta _{0}\text{.\label{pos generic op}}
\end{equation}%
With $\delta _{0}$ in hand, we perform our first deformation by applying the
main lemma to $\mathcal{U}_{1}.$ It follows from Lemma \ref{lem-C} that if
the deformation is small enough, then the resulting metric $g^{1}$ is $G$
--generic on $\mathcal{U}_{1}$. By combining (\ref{pos generic op}) with a
possible further restriction of the deformation size, we can also conclude
that the set $\mathcal{R}_{1}$ of $G$--rigid planes of $g^{1}$ is contained
in $\cup _{i=2}^{K}\mathcal{U}_{i}.$ As above, it follows from compactness
that there is a $\delta _{1}$ so that 
\begin{equation}
G^{g^{1}}|_{\mathcal{G}r\left( M\right) \setminus \cup _{i=2}^{K}\mathcal{U}%
_{i}}>\delta _{1},\text{\label{pos generic 1}}
\end{equation}%
where $G^{g^{1}}$ is the generic plane operator for $g^{1}.$

For the second step, apply the main lemma to $g^{1}$ on $\mathcal{U}_{2}$ to
obtain a deformation of metrics that, by Lemma \ref{lem-C}, are $G$--generic
on $\mathcal{U}_{2}$. By (\ref{pos generic 1}), if the deformation is small
enough, then the set $\mathcal{R}_{2}$ of $G$--rigid planes of $g^{2}$ is
contained in $\cup _{i=3}^{K}\mathcal{U}_{i}.$ Proceeding inductively, in
this manner, we obtain a $G$--generic metric that is as close as we please
to $g.$

%---------------------------------------------------------------------------------------------------------------------------------%


\begin{thebibliography}{99}
\bibitem{BLM} B. Basillo, C. Lee, and J. Malionek, \emph{Totally geodesic
surfaces in hyperbolic 3--manifolds: Algorithms and examples}, preprint.
https://arxiv.org/pdf/2403.12397.pdf

\bibitem{bryant} R. Bryant, \emph{\
http://mathoverflow.net/questions/209618/existence-of-totally-geodesic-hypersurfaces%
}.

\bibitem{Hirsch} M. Hirsch, \emph{Differential Topology}, Graduate Texts in
Mathematics, Springer-Verlag, 1994.

\bibitem{LytPetr} A. Lytchak and A. Petrunin, \emph{About every convex set
in generic Riemannian manifold}, Journal f\"{u}r die reine und angewandte
Mathematik (Crelles Journal), vol. 2022, no. 782, 2022, pp. 235-245.
https://doi.org/10.1515/crelle-2021-0058

\bibitem{MurphWilh} T. Murphy and F. Wilhelm,\emph{\ Random manifolds have
no totally geodesic submanifolds}, Michigan Mathematical Journal, 68 (2019),
323--335.

\bibitem{Perel1} G. Perelman, \emph{The entropy formula for the Ricci flow
and its geometric applications}, arXiv:math/0211159.

\bibitem{Perel2} G. Perelman, \emph{Ricci flow with surgery on
three-manifolds},. arXiv:math/0303109.

\bibitem{Perel3} G. Perelman, \emph{Finite extinction time for the solutions
to the Ricci flow on certain three-manifolds}, arXiv:math/0307245.

\bibitem{Pet} P. Petersen, \emph{Riemannian Geometry,} GTM Vol. 171 ,$2^{nd}$
Ed., New York: Springer Verlag, 2006.

\bibitem{spivak} M. Spivak, \emph{Differential Geometry}, Vol. III, Publish
or Perish, 1975.
\end{thebibliography}
\end{document}